\documentclass[oneside,english]{amsart}
\usepackage[T1]{fontenc}
\usepackage[latin9]{inputenc}
\usepackage{amstext}
\usepackage{amsthm}
\usepackage{amssymb}

\makeatletter
\numberwithin{equation}{section}
\numberwithin{figure}{section}
\numberwithin{table}{section}
\theoremstyle{plain}
\newtheorem{thm}{\protect\theoremname}[section]
\theoremstyle{definition}
\newtheorem{defn}[thm]{\protect\definitionname}
\theoremstyle{plain}
\newtheorem{lem}[thm]{\protect\lemmaname}
\theoremstyle{plain}
\newtheorem{cor}[thm]{\protect\corollaryname}

\makeatother

\usepackage{babel}
\providecommand{\corollaryname}{Corollary}
\providecommand{\definitionname}{Definition}
\providecommand{\lemmaname}{Lemma}
\providecommand{\theoremname}{Theorem}

\begin{document}
\title{Generalized central sets theorem for partial semigroups and VIP systems }
\author{Anik Pramanick, Md Mursalim Saikh}
\email{pramanick.anik@gmail.com}
\address{Department of Mathematics, University of Kalyani, Kalyani, Nadia-741235,
West Bengal, India.}
\email{mdmsaikh2016@gmail.com}
\address{Department of Mathematics, University of Kalyani, Kalyani, Nadia-741235,
West Bengal, India.}
\begin{abstract}
The Central sets theorem was first introduced by H. Furstenberg \cite{key-3}
in terms of Dynamical systems. Later Hindman and Bergelson extended
the theorem using Stone-\v{C}ech compactification $\beta\mathbb{N}$
of $\mathbb{N}.$ In \cite{key-12} algebraic characterisation of
Central sets were done for semigroup and equivalence of Dynamical
and Algebraic characterisations was shown. D. De, N. Hindman, and
D. Strauss proved a stronger version of Central sets theorem for semigroup.
D. Phulara genaralize that theorem for commutative semigroup taking
a sequence of Central sets. Recently J. Podder and S. Pal established
the Phulara type generalisation of Central sets theorem near zero
\cite{key-11}. We did this for arbitrary adequate partial semigroup
and VIP systems.
\end{abstract}

\maketitle

\section{Introduction}

In Ramsey theory Central sets Theorem has its own importance. After
the foundation of both van der Waerden\textquoteright s and Hindman\textquoteright s
theorem, an immediate question appears if one can find a joint extension
of both of these theorems. In \cite{key-3}, using the methods of
Topological dynamics, Furstenberg defined the notions of Central Sets
and proved that if $\mathbb{N}$ is finitely colored, then one of
the color classes is Central.

Here we mention some notational definitions  that we use through out
this article.
\begin{defn}
$\left(a\right)$ Given a set $A,\mathcal{P}_{f}\left(A\right)=\left\{ F:\phi\neq F\subseteq A\text{ \text{and \ensuremath{F} is finite}}\right\} $
\end{defn}

$\left(b\right)$ $\mathcal{J}_{m}=\left\{ t\in\mathbb{N}^{m}:t\left(1\right)<t\left(2\right)<...<t\left(m\right)\right\} .$

$\left(c\right)$ $\mathcal{F}_{d}=\left\{ A\subset\mathbb{N}:\mid A\mid\leq d\right\} $

$\left(d\right)$ Let $\left(H_{n}\right)_{n=1}^{\infty}$be a sequence
by, $FU\left(\left(H_{n}\right)_{n=1}^{\infty}\right)=\left\{ \sum_{n\in F}H_{n}:F\in\mathcal{P}_{f}\left(\mathbb{N}\right)\right\} $

$\left(e\right)$ $\left[n\right]$ =$\left\{ 1,2,...,n\right\} $,
$n\in\mathbb{N}$

$\left(f\right)$ let $\left(H_{n}\right)_{n=1}^{\infty},H_{n}\in\mathcal{P}_{f}\left(\mathbb{N}\right),By\text{ }H_{n}<H_{n+1}\text{ we mean }\max H_{n}<\min H_{n+1}$

$\left(g\right)$ An IP ring $\mathcal{F}^{\left(1\right)}$ is a
set of the form $\mathcal{F}^{\left(1\right)}=FU\left(\left(\alpha_{n}\right)_{n=1}^{\infty}\right)$
where $\left(\alpha_{n}\right)_{n=1}^{\infty}$ is a sequence of members
of $\mathcal{P}_{f}\left(\mathbb{N}\right)$ such that $\max\alpha_{n}<\min\alpha_{n+1}$
for each n.
\begin{thm}
\label{anik}Let $l\in\mathbb{N}$ and for each $i\in\left[l\right],$
let $\left(y_{i,n}\right)_{n=1}^{\infty}$be a sequence in $\mathbb{Z}$.
Let $C$ be a $central$ subset of $\mathbb{N}.$ Then there exists
sequences $\left(a_{n}\right)_{n=1}^{\infty}$in $\mathbb{N}$ and
$\left(H_{n}\right)_{n=1}^{\infty}$ in $\mathcal{P}_{f}\left(\mathbb{N}\right)$
such that
\end{thm}

$\left(1\right)$ for all $n,\max H<\min H_{n+1}$ and

$\left(2\right)$ for all $F\in\mathcal{P}_{f}\left(\mathbb{N}\right)$
and all $i\in\left[l\right],\sum_{n\in F}\left(a_{n}+\sum_{t\in H_{n}}y_{i,t}\right)\in C$.

Theorem \ref{anik} is the central sets theorem proved by Furstenberg
in 1981. Later in 1990 V. Bergelson and N. Hindman proved a different
but an equivalent version of the Central sets theorem.
\begin{thm}
Let $\left(S,+\right)$ be a commutative semigroup. Let $l\in\mathbb{N}$
and for each $i\in\left\{ 1,2,...,l\right\} $, let $\left(y_{i,n}\right)_{n=1}^{\infty}$be
a sequence in $S$. Let $C$ be a central subset of $S$. Then there
exist sequences $\left(a_{n}\right)_{n=1}^{\infty}$in $S$ and $\left(H_{n}\right)_{n=1}^{\infty}$
in $\mathcal{P}_{f}\left(\mathbb{N}\right)$ such that

$\left(1\right)$ for all $n,\text{\ensuremath{\max H_{n}}}<\min H_{n+1}$
and 
\end{thm}

$\left(2\right)$ for all $F\in\mathcal{P}_{f}\left(\mathbb{N}\right)$
and all $f:F\to\left\{ 1,2,...,l\right\} ,$
\[
\sum_{n\in F}\left(a_{n}+\sum_{t\in H_{n}}y_{f\left(i\right),t}\right)\in C.
\]

In 2008 D. De, N. Hindman and D. Strauss proved a stronger version
of central sets theorem.
\begin{thm}
\label{De}Let $\left(S,+\right)$ be a commutative semigroup and
let $C$ be a central subset of $S$. Then there exist functions $\alpha:\mathcal{P}_{f}\left(^{\mathbb{N}}S\right)\to S\text{ and }H:\mathcal{P}_{f}\left(S^{\mathbb{N}}\right)\to\mathcal{P}_{f}\left(\mathbb{N}\right)$
such that
\end{thm}

$\left(1\right)$ If $F,G\in\mathcal{P}_{f}\left(S^{\mathbb{N}}\right)$
and $F\subsetneq G$ then $\max H\left(F\right)<\min H\left(G\right)$
and

$\left(2\right)$ If $m\in\mathbb{N},G_{1},G_{2},....,G_{m}\in\mathcal{P}_{f}\left(S^{\mathbb{N}}\right)$;$G_{1}\subsetneq G_{2}\subsetneq....\subsetneq G_{m}$;
and for each $i\in\left\{ 1,2,....,m\right\} ,\left(y_{i,n}\right)\in G_{i},$
then 
\[
\sum_{i=1}^{m}\left(\alpha\left(G_{i}\right)+\sum_{t\in H\left(G_{i}\right)}y_{i,t}\right)\in C.
\]

In 2015 D. Phulara generalize the stronger version for commutative
semigroup. The theorem is the following
\begin{thm}
\label{Pu}Let $\left(S,+\right)$ be a commutative semigroup, let
$r$ be an idempotent in $J\left(S\right),$and let $\left(C_{n}\right)_{n=1}^{\infty}$
be a sequence of members of $r.$ There exists $\alpha:\mathcal{P}_{f}\left(\mathbb{^{\mathbb{N}}}S\right)\to S\text{ and }H:\mathcal{P}_{f}\left(\mathbb{^{\mathbb{N}}}S\right)\to\mathcal{P}_{f}\left(\mathbb{N}\right)$
such that 
\end{thm}

$\left(1\right)$ If $F,G\in\mathcal{P}_{f}\left(\mathbb{^{\mathbb{N}}}S\right)$
and $F\subsetneq G$ then $\max H\left(F\right)<\min H\left(G\right)$and

$\left(2\right)$Whenever $t\in\mathbb{N},G_{1},G_{2},....,G_{t}\in\mathcal{P}_{f}\left(^{\mathbb{N}}S\right)$;$G_{1}\subsetneq G_{2}\subsetneq....\subsetneq G_{t}$;$\mid G_{1}\mid=m$
and for each $i\in\left\{ 1,2,....,n\right\} ,f_{i}\in G_{i},$ then

$\sum_{i=1}^{t}\left(\alpha\left(G_{i}\right)+\sum_{s\in H\left(G_{i}\right)}f_{i}\left(s\right)\right)\in C_{m}.$

Later in 2021 N. Hindman and K. Pleasant proved the central sets theorem
for adequate partial semigroup in \cite{key-6}. Here we generalize
the theorem by K. Pleasant and N. Hindman in D. Phulara's way. Apart
from that we generalize the central sets theorem for VIP system in
commutative adequate partial semigroup. Now we briefly discuss VIP
system here. VIP system is polynomial type configuration.
\begin{defn}
Let $\left(G,+\right)$ be an abelian group. A sequence $\left(v_{\alpha}\right)_{\alpha\in\mathcal{P}_{f}\left(\mathbb{N}\right)}$
in $G$ is called a VIP system if there exists some non-negative integer
d (the least such $d$ is called the degree of the system) such that
for every pairwise disjoint $\alpha_{0},\alpha_{1},...,\alpha_{d}\in\mathcal{P}_{f}\left(\mathbb{N}\right)$
we have $\sum_{t=1}^{d+1}\left(-1\right)^{t}\sum_{\mathcal{B\in}\left[\left\{ \alpha_{0},\alpha_{1},...,\alpha_{d}\right\} \right]^{t}}v_{\cup\mathcal{B}}=0$.
\end{defn}

In their paper \cite{key-5} generalize this notion for partial semigroup.
They defined the VIP system for partial semigroup in the following
way.
\begin{defn}
Let $\left(S,+\right)$ be a commutative partial semigroup. Let $\left(v_{\alpha}\right)_{\alpha\in\mathcal{P}_{f}\left(\mathbb{N}\right)}$
be a sequence in $S.$ $\left(v_{\alpha}\right)_{\alpha\in\mathcal{P}_{f}\left(\mathbb{N}\right)}$
is called a VIP system if there exists some $d\in\mathbb{N}$ and
a function from $\mathcal{F}_{d}$ to $S\cup\left\{ 0\right\} ,$written
$\gamma\to m_{\gamma},\gamma\in$$\mathcal{F}_{d}$ , such that
\end{defn}

\[
v_{\alpha}=\sum_{\gamma\subseteq\alpha,\gamma\in\mathcal{F}_{d}}m_{\gamma}\text{ for all }\alpha\in\mathcal{P}_{f}\left(\mathbb{N}\right).\text{ ( In particular, the sum is always defined)}
\]
 The sequence $\left(m_{\gamma}\right)_{\gamma\in\mathcal{F}_{d}}$
is said to generate the VIP system $\left(v_{\alpha}\right)_{\alpha\in\mathcal{P}_{f}\left(\mathbb{N}\right)}.$

Later they proved the Central sets theorem for VIP systems of commutative
adequate partial semigroup. 

\section{ALGEBRAIC BACKGROUND}

Here we briefly discuss about the Stone-\v{C}ech compactification
$\beta S$ of a semigroup $S$. $\beta S$ is the collection of all
ultrafilters on $S$ and we identify the principal ultrafilters with
the points of $S.$ For $A\subseteq S$, $\overline{A}=\left\{ p\in\beta S:A\in p\right\} .$
The set $\left\{ \overline{A}:A\subseteq S\right\} $ forms a basis
for the compact Hausdorff topology on $\beta S$. For more information
about $\beta S$ readers are requested to see \cite{key-7}. We will
discuss about partial semigroup here.
\begin{defn}
A partial semigroup is a pair $\left(S,*\right)$ where $*$ maps
a subset of $S\times S$ to $S$ and for all $a,b,c,\in S$, $\left(a*b\right)*c=a*\left(b*c\right)$
in the sense that if either side is defined, then so is the other
and they are equal.
\end{defn}

For examples of partial semigroups readers are requested to go through
\cite{key-8}. 
\begin{defn}
Let $\left(S,*\right)$ be a partial semigroup.

$\left(a\right)$ For $s\in S,\varphi\left(s\right)=\left\{ t\in S:s*t\text{ }\text{is\text{ }defined}\right\} $

$\left(b\right)$ For $H\in\mathcal{P}_{f}\left(S\right),\sigma\left(H\right)=\cap_{s\in H}\varphi\left(s\right)$

$\left(c\right)\text{ }\sigma\left(\phi\right)=S$

$\left(d\right)$ For $s\in S$ and $A\subseteq S,\text{ }s^{-1}A=\left\{ t\in\varphi\left(s\right):s*t\in A\right\} $

$\left(e\right)$ $\left(S,*\right)$ is adequate if and only if $\sigma\left(H\right)\neq\phi$
for all $H\in\mathcal{P}_{f}\left(S\right)$.
\end{defn}

\begin{lem}
Let $\left(S,*\right)$ be a partial semigroup, let $A\subseteq S$
and let $a,b,c\in S.$ Then $c\in b^{-1}\left(a^{-1}A\right)\iff b\in\varphi\left(a\right)and\text{ }c\in\left(a*b\right)^{-1}A.$
In particular, if $b\in\varphi\left(a\right)$, then $b^{-1}\left(a^{-1}A\right)=\left(a*b\right)^{-1}A$.
\end{lem}

\begin{proof}
\cite{key-5}, Lemma $2.3$
\end{proof}
We are specifically interested in adequate partial semigroups as they
lead to an interesting sub semigroup of $\beta S$. This subsemigroup
is itself a compact right topological semigroup and is defined next.
\begin{defn}
Let $\left(S,*\right)$ be a partial semigroup. Then 

$\delta S=\cap_{x\in S}\overline{\varphi\left(x\right)}=\cap_{H\in\mathcal{P}_{f}\left(S\right)}\overline{\sigma\left[H\right]}$

Notice that $\delta S\neq\phi$ when the partial semigroup $S$ is
adequate and for $S$ being semigroups $\delta S=\beta S$. 

For $\left(S,.\right)$ be a semigroup, $A\subseteq S,a\in S,\text{ }and\text{ }p,q\in\beta S,$
then $A\in a.q\iff a^{-1}A\in q$

and

$A\in p.q\iff\left\{ a\in S:a^{-1}A\in q\right\} \in p$

Now we extend this notion for partial operation $*$. 

Let $\left(S,*\right)$ be an adequate partial semigroup.

$\left(a\right)$ For $a\in S$ and $q\in\overline{\varphi\left(a\right)},a*q=\left\{ A\subseteq S:a^{-1}A\in q\right\} .$

$\left(b\right)$ For $p\in\beta S$ and $q\in\delta S,p*q=\left\{ A\subseteq S:\left\{ a\in S:a^{-1}A\in q\right\} \in p\right\} $.
\end{defn}

\begin{lem}
2.6. Let $\left(S,*\right)$ be an adequate partial semigroup.

$\left(i\right)$ If $a\in S$ and $q\in\overline{\varphi\left(a\right),}then\text{ }a*q\in\beta S.$

$\left(ii\right)$ If $p\in\beta S$ and $q\in\delta S$, then $p*q\in\beta S.$

$\left(iii\right)$ Let $p\in\beta S,q\in\delta S,$and $a\in S.$Then
$\varphi\left(a\right)\in p*q$ if and only if $\varphi\left(a\right)\in p.$

$\left(iv\right)$ If $p,q\in\delta S,$then $p*q\in\delta S$.
\end{lem}

\begin{proof}
\cite{key-5}, Lemma $2.7$
\end{proof}
\begin{lem}
Let $\left(S,*\right)$ be an adequate partial semigroup and let $q\in\delta S.$
Then the function $\rho_{q}:\beta S\to\beta S$ defined by $\rho_{q}\left(p\right)=p*q$
is continuous.
\end{lem}

\begin{proof}
\cite{key-5},Lemma $2.8$.
\end{proof}
\begin{thm}
Let $\left(S,*\right)$ be an adequate partial semigroup. Then $\left(\delta S,*\right)$
is a compact Hausdorff right topological semigroup.
\end{thm}

\begin{proof}
\cite{key-5}, Theorem $2.10$
\end{proof}
\begin{thm}
Let $p=p*p\in\delta S$ and let $A\in p.$ Then $A^{*}=\left\{ x\in A:x^{-1}A\in p\right\} $
\end{thm}

For an idempotent $p\in\delta S$ and $A\in p,$then $A^{*}\in p$.
\begin{lem}
\label{lem1}Let $p=p*p\in\delta S,$let $A\in p,$let $x\in A^{*}.$
Then $x^{-1}\left(A^{*}\right)\in p.$
\end{lem}

\begin{defn}
Let $\left(S,*\right)$be a partial semigroup and let $A\subseteq S.$Then
$A$ is syndetic if and only if there is some $H\in\mathcal{P}_{f}\left(S\right)$
such that $\sigma\left(H\right)\subseteq\cup_{t\in H}t^{-1}A.$
\end{defn}

\begin{lem}
Let $\left(S,*\right)$ be an adequate partial semigroup and let $A\subseteq S.$Then
$A$ is syndetic if and only if there exists $H\in\mathcal{P}_{f}\left(S\right)$
such that $\delta S\subseteq\cup_{t\in H}\overline{t^{-1}A}.$ 
\end{lem}

\begin{defn}
$K\left(\delta S\right)=\left\{ A:A\text{ is a minimal left ideal in }\delta S\right\} .$
\end{defn}

\begin{thm}
Let $\left(S,*\right)$ be an adequate partial semigroup and let $p\in\delta S.$
The following statements are equivalent.
\end{thm}

$\left(a\right)$ $p\in K\left(\delta S\right).$

$\left(b\right)$ For all $A\in p$, $\left\{ x\in S:x^{-1}A\in p\right\} $is
syndetic.

$\left(c\right)$ For all $q\in\delta S,p\in\delta S*q*p.$
\begin{proof}
\cite{key-5}, Theorem $2.15$
\end{proof}
\begin{defn}
Let $\left(S,*\right)$ be an adequate partial semigroup and let $A\subseteq S.$
\end{defn}

$\left(a\right)$ The set $A$ is $piecewise$ syndetic in $S$ if
and only if $\overline{A}\cap K\left(\delta S\right)\neq\phi.$

$\left(b\right)$ The set $A$ is $central$ in $S$ if and only if
there is some idempotent $p$ in $K\left(\delta S\right)$ such that
$A\in p.$

$\left(c\right)$ A set $A\subseteq S$ is a $J$-set if and only
if for all $F\in\mathcal{P}_{f}\left(\mathcal{F}\right)$ and all
$L\in\mathcal{P}_{f}\left(S\right),$ there exists $m\in N,a\in S^{m+1},$and
$t\in\mathcal{J}_{m}$ such that for all $f\in F$,
\[
\text{\ensuremath{\left(\prod_{i=1}^{m}a\left(i\right)*f\left(t\left(i\right)\right)\right)}*\ensuremath{\left(a\left(m+1\right)\right)}\ensuremath{\ensuremath{\in}A\ensuremath{\cap\sigma\left(L\right)}}}
\]

$\left(d\right)$ $J\left(S\right)=\left\{ p\in\delta S:\left(\forall A\in P\right)\left(A\text{ }is\text{ }a\text{ }J\text{-set}\right)\right\} .$
\begin{lem}
Let $\left(S,*\right)$ be an adequate partial semigroup and let $A\subseteq S$
be $piecewise\text{ }syndetic.$ There exists $H\in\mathcal{P}_{f}\left(S\right)$
such that for every finite nonempty set $T\subseteq\sigma\left(H\right),$there
exists $x\in\sigma\left(T\right)$ such that $T*x\subseteq\cup_{t\in H}t^{-1}A.$
\end{lem}

Now we will mention one of the crucial concept adequate sequence for
partial semigroup. 
\begin{defn}
.Let $\left(S,*\right)$ be an adequate partial semigroup and let
$f$ be a sequence in $S.$ Then $f$ is adequate if and only if
\end{defn}

$\left(1\right)$ for each $H\in\mathcal{P}_{f}\left(\mathbb{N}\right)$,
$\prod_{t\in H}f\left(t\right)$ is defined and

$\left(2\right)$ for each $F\in\mathcal{P}_{f}\left(S\right)$, there
exists $m\in\mathbb{N}$ such that 
\[
FP\left(\left(f\left(t\right)\right)_{t=m}^{\infty}\right)\subseteq\sigma\left(F\right).
\]

\begin{defn}
Let $\left(S,*\right)$ be an adequate partial semigroup. Then 
\[
\mathcal{F}=\left\{ f:f\text{\ensuremath{}is an adequate sequence in }S\right\} .
\]
\end{defn}

\section{PHULARA VERSION OF CENTRAL SETS THEOREM FOR ADEQUATE PARTIAL SEMIGROUP}

In \cite{key-8}, Jillian McLeod establishes a version of Theorem
\ref{anik} valid for commutative adequate partial semigroups. In
\cite{key-10}, Kendra Pleasant and in \cite{key-4}, Arpita Ghosh,
independently but later, prove a version of Theorem \ref{De} for
commutative adequate partial semigroups. In \cite{key-9} Plulara
generalized Central sets theorem for commutative semigroup. In this
paper, we show that Theorem \ref{Pu} remains valid for arbitrary
adequate partial semigroups. To prove that we need the following lemma.
\begin{lem}
\label{lem2}Let $\left(S,*\right)$ be an adequate partial semigroup
and let $A$ be a $J$-set in $S$. Let $r\in\mathbb{N},$ let $F\in\mathcal{P}_{f}\left(\mathcal{F}\right),$and
let $L\in\mathcal{P}_{f}\left(S\right).$ There exists $m\in\mathbb{N},a\in S^{m+1},and\text{ }t\in\mathcal{J}_{m}$
such that $t\left(1\right)>r$ and for all $f\in F,$
\end{lem}

$\left(\prod_{i=1}^{m}a\left(i\right)*f\left(t\left(i\right)\right)*a\left(m+1\right)\in A\cap\sigma\left(L\right)\right).$
\begin{proof}
\cite{key-6}, Lemma $3.5$
\end{proof}
Now we state the main theorem of this section.
\begin{thm}
Let $\left(S,*\right)$ be an adequate partial semigroup and let $r$
be an idempotent in $J\left(S\right)$ and let $\left(C_{n}\right)_{n=1}^{\infty}$
be a sequence of members of $r$. Then there exists functions $m^{*}:\mathcal{P}_{f}\left(\mathcal{F}\right)\rightarrow\mathbb{N},$and
$\alpha\in\times_{F\in\mathcal{P}_{f}\left(\mathcal{F}\right)}S^{m\left(F\right)+1}$
and $\mathcal{T}\in\times_{F\in\mathcal{P}_{f}\left(\mathcal{F}\right)}\mathcal{J}_{m^{*}\left(F\right)}$
such that 

1. If $F,G\in\mathcal{P}_{f}(\mathcal{F}),G\subset F$, then $\mathcal{T}(G)(m^{*}(G))<\mathcal{T}(F)(1).$
and 

2. If $m\in\mathbb{N}$ and $G_{1},G_{2},...,G_{s}\in\mathcal{P}_{f}\left(\mathcal{F}\right),G_{1}\subset G_{2}\subset.....\subset G_{s},\mid G_{1}\mid=m$,
and $f_{i}\in G_{i},i=1,2,...,s.$ then

\[
\prod_{i=1}^{s}\left(\left(\prod_{j=1}^{m^{*}(G_{i})}\alpha(G_{i})(j)*f_{i}(\mathcal{T}(G_{i})(j)\right)*\alpha(G_{i})(m^{*}(G_{i})+1)\right)\in C_{m}
\]
\end{thm}

\begin{proof}
We assume $C_{n+1}\subseteq C_{n}$ for each $n\in\mathbb{N}$ ( If
not, consider $B_{n}=\cap_{i=1}^{n}C_{i}$ , so $B_{n+1}\subseteq B_{n}$
) . For each $n\in\mathbb{N}$, let $C_{n}^{*}=\left\{ x\in C_{n}:x^{-1}C_{n}\in r\right\} .$
Then $C_{n}^{*}\in r$ and by Lemma \ref{lem1} if $x\in C_{n}^{*}$
then $x^{-1}C_{n}^{*}\in r.$ 

Now we use induction hypothesis to prove the statement. 

Let $\mid F\mid=1$ and $F=\left\{ f\right\} $. Then statement 1
is vacuously true. Pick $d\in S$ and let $L=\left\{ d\right\} .$
Pick $m\in\mathbb{N}$, $a\in S^{m+1}$ and $t\in\mathcal{J}_{m}$
such that $\prod_{i=1}^{m}a\left(i\right)*f\left(t\left(i\right)\right)*a\left(m+1\right)\in C_{1}^{*}$
Let $m^{*}\left(F\right)=m$ , $\alpha\left(F\right)=a$ , $\mathcal{T}(F)=t$
( By Lemma \ref{lem2}). Now let the statement is true for all $F$
with $\mid F\mid<n$ , $n\in\mathbb{N}.$ Let 
\[
M_{m}=\left\{ \begin{array}{c}
\prod_{i=1}^{s}\left(\left(\prod_{j=1}^{m^{*}(G_{i})}\alpha(G_{i})(j)*f_{i}(\mathcal{T}(G_{i})(j)\right)*\alpha(G_{i})(m^{*}(G_{i})+1)\right):\\
s\in\mathbb{N},\mid G_{1}\mid=m\text{ and for each }i\in\left\{ 1,2,...,s\right\} ,\\
f_{i}\in G_{i},\phi\subsetneq G_{1}\subsetneq G_{2}\subsetneq....\subsetneq G_{s}\subsetneq F
\end{array}\right\} 
\]
where $m\in\left\{ 1,2,....,n-1\right\} $. Then $M_{m}$ is finite
and by induction $M_{m}\subseteq C_{m}^{*}$.

Let $A=C_{n}^{*}\cap\left(\cap_{m=1}^{n-1}\left(\bigcap_{x\in M_{m}}\left(x^{-1}C_{m}^{*}\right)\right)\right)$.

Then $A\in r,$ so $A$ is a $J$-set. 

Let $d=max\left\{ \mathcal{T}\left(G\right)\left(m^{*}\left(G\right)\right):\phi\neq G\subsetneq F\right\} $.
By Lemma \ref{lem2} pick $k\in\mathbb{N}$,$a\in S^{k+1}$and $t\in\mathcal{J}_{k}$
such that $t\left(1\right)>d$ and for all $\text{f\ensuremath{\in F}}$.
\[
\prod_{i=1}^{k}a\left(i\right)*f\left(t\left(i\right)\right)*a\left(k+1\right)\in A
\]
Define $m\left(F\right)=k$, $\alpha\left(F\right)=a$,$\mathcal{\text{ }T}\left(F\right)=\text{t}$.
So (1) is satisfied. To verify hypothesis (2) assume $s=1$, then
$G_{1}=F$ and $m=n$, so 
\[
\prod_{i=1}^{m}a\left(i\right)*f_{s}\left(t\left(i\right)\right)*a\left(m+1\right)\in A\subseteq C_{m}^{*}
\]
Let $y=\prod_{i=1}^{s-1}\left(\left(\prod_{j=1}^{m^{*}(G_{i})}\alpha\left(G_{i}\right)\left(j\right)*f_{i}(\mathcal{T}(G_{i})(j)\right)*\alpha(G_{i})(m^{*}(G_{i})+1)\right)$.
Then $y\in M_{m}$ , so $\prod_{i=1}^{m}a\left(i\right)*f_{s}\left(t\left(i\right)\right)*a\left(m+1\right)\in y^{-1}C_{m}^{*}$
therefore 
\[
\prod_{i=1}^{s}\left(\left(\prod_{j=1}^{m^{*}(G_{i})}\alpha(G_{i})(j)*f_{i}(\mathcal{T}(G_{i})(j)\right)*\alpha(G_{i})(m^{*}(G_{i})+1)\right)\in C_{m}^{*}\subseteq C_{m}
\]
 .
\end{proof}
\begin{cor}
Let $\left(S,*\right)$ be a commutative adequate partial semigroup
and let $r$ be an idempotent in $J\left(S\right)$ and let $\left(C_{n}\right)_{n=1}^{\infty}$
be a sequence of members of $r$. Then there exists functions
\end{cor}

$\gamma:\mathcal{P}_{f}\left(\mathcal{F}\right)\to S$ and $H:\mathcal{P}_{f}\left(\mathcal{F}\right)\to\mathcal{P}_{f}\left(\mathbb{N}\right)$
such that

$\left(1\right)$ if $F,G\in\mathcal{P}_{f}\left(\mathcal{F}\right)$
and $G\subsetneq F$ then $maxH\left(G\right)<minH\left(F\right)$and

$\left(2\right)$ if $n\in\mathbb{N},G_{1},G_{2},....,G_{n}\in\mathcal{P}_{f}\left(\mathcal{F}\right)$;$G_{1}\subsetneq G_{2}\subsetneq....\subsetneq G_{n}$;$\mid G_{1}\mid=m$
and for each $i\in\left\{ 1,2,....,n\right\} ,f_{i}\in G_{i},$ then
$\prod_{i=1}^{n}\left(\gamma\left(G_{i}\right)*\prod_{t\in H\left(G_{i}\right)}f_{i}\left(t\right)\right)\in C_{m}^{*}$
\begin{proof}
Let $m^{*},\alpha$ and $\mathcal{T}$ be as guaranteed by previous
theorem. For $F\in\mathcal{P}_{f}\left(\mathcal{F}\right),$Let $\gamma\left(F\right)=\prod_{j=1}^{m^{*}\left(F\right)+1}\alpha\left(F\right)\left(j\right)$
and let $H\left(F\right)=\left\{ \mathcal{T}\left(F\right)\left(j\right):j\in\left\{ 1,2,...,m^{*}\left(F\right)\right\} \right\} $.
\end{proof}
\begin{cor}
\label{cor}Let $\left(S,*\right)$ be a non trivial commutative adequate
partial semigroup,let $r$ be an idempotent in $J(S)$, Let $\left(C_{n}\right)_{n=1}^{\infty}$
be a sequence of members of $r$, let $k\in\mathbb{N}$ and for each
$l\in\left\{ 1,2,...,k\right\} $,let $\left(y_{l,n}\right)_{n=1}^{\infty}$be
an adequate sequence in $S$. Then there exists a sequence $\left(a_{n}\right)_{n=1}^{\infty}$in
$S$ and a sequence $\left(H_{n}\right)_{n=1}^{\infty}$ in $\mathcal{P}_{f}\left(\mathbb{N}\right)$
with $\max H_{n}<\min H_{n+1}$for each $n$ such that for $l\in\left\{ 1,2,....,k\right\} $and
for each $F\in\mathcal{P}_{f}\left(\mathbb{N}\right)$ with $m=\min F$
one has
\end{cor}

$\prod_{n\in F}\left(a_{n}*\prod_{t\in H_{n}}y_{l,t}\right)\in C_{m}$.
\begin{proof}
We may assume that $C_{n+1}\subseteq C_{n}$ for each $n\in\mathbb{N}$.
Pick $\gamma$ and $H$ as guarnteed by previous corollary. Choose
$\gamma_{u}\in\mathcal{F}$\textbackslash$\left\{ \left(y_{1,n}\right)_{n=1}^{\infty},\left(y_{2,n}\right)_{n=1}^{\infty},...,\left(y_{k,n}\right)_{n=1}^{\infty}\right\} $
such that $\gamma_{u}\neq\gamma_{v}$ if $u\neq v$ which we can do
because S is non trivial. For $u\in\mathbb{N}$ let 
\[
G_{u}=\text{\ensuremath{\left\{  \left(y_{1,n}\right)_{n=1}^{\infty},\left(y_{2,n}\right)_{n=1}^{\infty},...,\left(y_{k,n}\right)_{n=1}^{\infty}\right\} } \ensuremath{\cup}}\left\{ \gamma_{1},\gamma_{2},...,\gamma_{u}\right\} 
\]
Let $a_{u}=\gamma\left(G_{u}\right)$ and $H_{u}=H\left(G_{u}\right)$.
Let $l\in\left\{ 1,2,...,k\right\} $ and let $F\in\mathcal{P}_{f}\left(\mathbb{N}\right)$
be enumerated in order as $\left\{ n_{1},n_{2},...n_{s}\right\} $so
that $m=n_{1}$ then $G_{m}=G_{n_{1}}\subsetneq G_{n_{2}}\subsetneq...\subsetneq G_{n_{s}}$.
Also for each $i\in\left\{ 1,2,...,s\right\} $,$\left(y_{l,t}\right)_{t=1}^{\infty}\in G_{n_{i}}$
and $\mid G_{n_{1}}\mid=m+k$ , so $\prod_{n\in F}\left(a_{n}*\prod_{t\in H_{n}}y_{l,t}\right)=\prod_{i=1}^{s}\left(\gamma\left(G_{n_{i}}\right)*\prod_{t\in H\left(G_{n_{i}}\right)}y_{l,t}\right)\in C_{m+k}\subseteq C_{m}$.
\end{proof}
Now we will see some combinatorial applications.
\begin{defn}
Let $u,v\in\mathbb{N}$ and let $A$ be a $u\times v$ matrix with
entries from $\mathbb{Q}.$ Then $A$ satisfies the first entries
conditions if and only if no row of $A$ is $\overrightarrow{0}$
and whenever $i,j\in\left\{ 1,2,...,u\right\} $ and $k=min\left\{ t\in\left\{ 1,2,...,v\right\} :a_{i,t}\neq0\right\} =min\left\{ t\in\left\{ 1,2,...,v\right\} :a_{j,t}\neq0\right\} ,$then
$a_{i,t}=a_{j,t}>0.$ An element $b\in\mathbb{Q}$ is a first entry
of $A$ if and only if there is some row $i$ of $A$ such that $b=a_{i,k}\text{ where }k=min\left\{ t\in\left\{ 1,2,...,v\right\} :a_{i,t}\neq0\right\} .$

If $A$ satisfies the first entries condition , we say that $A$ is
a first entries matrix. 
\end{defn}

\begin{thm}
Let $S$ be a commutative adequate partial semigroup and $A$ be a
$u\times v$ matrix which satisfies the first entries condition. L$ $et
$\left(C_{n}\right)_{n=1}^{\infty}$ be central subsets of $S$ .
Assume that for each first entry $c$ of $A$, and for each $n\in\mathbb{N}$,
$cS\cap C_{n}$ is a central{*} set. Then for each $i=1,2,\cdots,v$
there exists adequate sequence in $S$ $\left(x_{i,n}\right)_{n=1}^{\infty}$
such that for every $F\in\mathcal{P}_{f}\left(\mathbb{N}\right)$
with $min\,F=m$, we have $A\overrightarrow{x}_{F}\in\left(C_{m}\right)^{u},$where
$\overrightarrow{x}_{F}\in S^{v}$ .
\[
\overrightarrow{x}_{F}=\left(\begin{array}{c}
\underset{n\in F}{\prod}x_{1,n}\\
\underset{n\in F}{\prod}x_{2,n}\\
\vdots\\
\underset{n\in F}{\prod}x_{v,n}
\end{array}\right).
\]
\end{thm}

\begin{proof}
We can assume that $C_{n+1}\subseteq C_{n}$. First we take $v=1$.
We can assume that $A$ has no repeated rows. In that case $A=\left(c\right)$
for some $c\in\mathbb{N}$ such that $cS$ is a central{*} set and
$\left(C_{n}\cap cS\right)_{n=1}^{\infty}$ is a sequence of members
of $p\in K(\delta S)$, satisfying $C_{n+1}\cap cS\subseteq C_{n}\cap cS$. 

Since we are in the base case, i.e. $v=1$, by Corollary \ref{cor},
we have adequate sequences $\left(a_{n}\right)_{n=1}^{\infty}$and
$\left(y_{1,n}\right)_{n=1}^{\infty}$ in $S$ with $y_{1,n}=0$ for
all $n$, such that $\prod_{n\in F}a_{n}\in C_{m}\cap cS$ where $min\,F=m$.
We choose $cx_{1,n}=a_{n}$. So the sequence $\left(x_{1,n}\right)_{n=1}^{\infty}$
is as required.

Now assume $v\in\mathbb{N}$ and the theorem is true for $v$. Let,
$A$ be a $u\times(v+1)$ matrix which satisfies the first entries
condition$ $, and assume that for every first entry $c$ of $A$,
$C_{n}\cap cS$ is a central{*} set for all $n$. By rearranging rows
of $A$ and adding additional rows of $A$ if needed, we may assume
that we have some $t\in\{1,2,..,u-1\}$ and $d\in\mathbb{N}$ such
that 
\[
a_{i,1}=\begin{cases}
\begin{array}{c}
0\\
d
\end{array} & \begin{array}{c}
\text{if}\:i\leq t\\
\text{if}\:i>t
\end{array}\end{cases}.
\]

So the matrix in block form looks like 
\[
\left(\begin{array}{cc}
\bar{0} & \mathcal{B}\\
\bar{d} & *
\end{array}\right)
\]
 where $\mathcal{B}$ is a $t\times v$ matrix with entries $b_{i,j}=a_{i,j+1}$.
So by inductive hypothesis we can choose $\left(w_{i,n}\right)_{n=1}^{\infty},i=\left\{ 1,2,\cdots,v\right\} $
for the matrix $\mathcal{B}$.

Let for each $i\in\left\{ t+1,t+2,\cdots,u\right\} $ and each $n\in\mathbb{N}$,
\[
y_{i,n}=\prod_{j=2}^{v+1}a_{i,j}\cdotp w_{j-1,n}.
\]

Now we have that $\left(C_{n}\cap dS\right)_{n=1}^{\infty}$ is a
sequence of members of $p\in K(\delta S)$. So by Corollary\ref{cor}
we can choose $\left(a_{n}\right)_{n=1}^{\infty}$ in $\mathcal{T}$
and $\left(H_{n}\right)_{n=1}^{\infty}$ in $\mathcal{P}_{f}\left(\mathbb{N}\right)$
such that $max\,H_{n}<min\,H_{n+1}$ for each $n\in\mathbb{N}$and
for each $i\in\left\{ t+1,t+2,\cdots,u\right\} $ and for all $F\in\mathcal{P}_{f}\left(\mathbb{N}\right)$
with $min\,F=m$, then 
\[
\prod_{n\in F}\left(a_{n}\prod_{s\in H_{n}}y_{i,s}\right)\in C_{m}\cap dS.
\]

In particular if $F=\left\{ n\right\} $ then pick $x_{1,n}\in S$
such that $a_{n}=d\cdotp x_{1,n}$. For $j\in\left\{ 2,3,\cdots,v+1\right\} $,
define $x_{j,n}=\prod_{s\in H_{n}}w_{j-1,s}$. The proof will be done
if we can show that $\left(x_{j,n}\right)_{n=1}^{\infty}$ are the
required sequences. So we need to show that for each $i\in\left\{ 1,2,\cdots,u\right\} $,
\[
\prod_{j=1}^{v+1}a_{i,j}\prod_{n\in F}x_{j,n}\in C_{m}.
\]

If $i\leq t$, then, 
\[
\begin{array}{cc}
\prod_{j=1}^{v+1}\left(a_{i,j}\prod_{n\in F}x_{j,n}\right) & =\prod_{j=2}^{v+1}\left(a_{i,j}\prod_{n\in F}\prod_{s\in H_{n}}w_{j-1,s}\right)\\
 & =\prod_{j=1}^{v}\left(b_{i,j}\prod_{s\in H}w_{j,s}\right)
\end{array}
\]
 where $H=\bigcup_{n\in F}H_{n}$. Let $m'=min\,H$, then $m'\geq m$
due to the condition that $max\,H_{n}<min\,H_{n+1}$ for each $n\in\mathbb{N}$.
Now by induction hypothesis we have, 
\[
\prod_{j=1}^{v+1}\left(a_{i,j}\prod_{n\in F}x_{j,n}\right)=\prod_{j=1}^{v}\left(b_{i,j}\prod_{s\in H}w_{j,s}\right)\in C_{m'}\subseteq C_{m}.
\]
 For the case $i>t$, 
\[
\begin{array}{cc}
\prod_{j=1}^{v+1}\left(a_{i,j}\prod_{n\in F}x_{j,n}\right) & =a_{i,1}\prod_{n\in F}x_{i,n}\prod_{j=2}^{v+1}\left(a_{i,j}\prod_{n\in F}\prod_{s\in H_{n}}w_{j-1,s}\right)\\
 & =d\prod_{n\in F}x_{1,n}\prod_{j=2}^{v+1}\left(a_{i,j}\prod_{n\in F}\prod_{s\in H_{n}}w_{j-1,s}\right)\\
 & =\prod_{n\in F}dx_{1,n}\prod_{n\in F}\prod_{s\in H_{n}}\prod_{j=2}^{v+1}a_{i,j}w_{j-1,s}\\
 & =\prod_{n\in F}\left(a_{n}\prod_{s\in H_{n}}y_{i,s}\right)\in C_{m}
\end{array}
\]
 and the theorem is done.
\[
\overrightarrow{x}_{F}=\left(\begin{array}{c}
\underset{n\in F}{\prod}x_{1,n}\\
\underset{n\in F}{\prod}x_{2,n}\\
\vdots\\
\underset{n\in F}{\prod}x_{v,n}
\end{array}\right).
\]
\end{proof}

\section{PHULARA VERSION OF CENTRAL SETS THEOREM FOR VIP SYSTEMS IN PARTIAL
SEMIGROUP}

Now we concentrate on a special class of finite families of VIP systems
and proceed for further generalization of Central sets Theorem.
\begin{defn}
Let $\left(S,+\right)$ be commutative adequate partial semigroup.
A finite set $\left\{ \left(v_{\alpha}^{\left(i\right)}\right)_{\alpha\in\mathcal{P}_{f}\left(\mathbb{N}\right)}:1\leq i\leq k\right\} $
of VIP systems is said to be $adequate$ if there exists $d,t\in\mathbb{N},$a
set $\left\{ \left(m_{\gamma}\right)_{\gamma\in\mathcal{F}_{d}}:i\in\left[k\right]\right\} ,$
a set of VIP systems 
\[
\left\{ \left(u_{\alpha}^{\left(i\right)}=\sum_{\gamma\subseteq\alpha,\gamma\in\mathcal{F}_{d}}n_{\gamma}^{\left(i\right)}\right)_{\alpha\in\mathcal{P}_{f}\left(\mathbb{N}\right)}:i\in\left[t\right]\right\} ,
\]
 and sets $E_{1},E_{2},...,E_{k}\subseteq\left\{ 1,2,...,t\right\} $
such that:
\end{defn}

$\left(1\right)$ For each $i\in\left\{ 1,2,...,k\right\} ,\left(m_{\gamma}\right)_{\gamma\in\mathcal{F}_{d}}$
generates $\left(v_{\alpha}^{\left(i\right)}\right)_{\alpha\in\mathcal{F}}$.

$\left(2\right)$ For every $H\in\mathcal{P}_{f}\left(S\right),$there
exists $m\in\mathbb{N}$ such that for every $l\in\mathbb{N}$ and
pairwise distinct $\gamma_{1},\gamma_{2},...,\gamma_{l}\in\mathcal{F}_{d}$
with each 
\[
\gamma_{i}\nsubseteq\left\{ 1,2,...,m\right\} ,\sum_{i=1}^{t}\sum_{j=1}^{l}n_{\gamma_{j}}^{\left(i\right)}\in\sigma\left(H\right)\cup\left\{ 0\right\} .
\]
( In particular, the sum is defined)

$\left(3\right)$ $\text{\ensuremath{m_{\gamma}^{\left(i\right)}}=\ensuremath{\sum_{t\in E_{i}}n_{\gamma}^{\left(t\right)}}}$
for all $i\in\left\{ 1,2,...,k\right\} $ and all $\gamma\in\mathcal{F}_{d}$.
\begin{defn}
Let $\left(S,+\right)$ be a commutative adequate partial semigroup
and let $\mathcal{A}\subseteq\mathcal{P}_{f}\left(S\right).\mathcal{A}$
is said to be $adequately\text{ }partition\text{ }regular$ if for
every finite subset $H$ of $S$ and every $r\in\mathbb{N},$ there
exists a finite set $F\subseteq\sigma\left(H\right)$ having the property
that if $F=\cup_{i=1}^{r}C_{i}$ then for some $j\in\left\{ 1,2,...,r\right\} ,C_{j}$
contains a member of $\mathcal{A}.$ $\mathcal{A}$ is said to be
$shift\text{ }invariant$ if for all $A\in\mathcal{A}$ and all $x\in\sigma\left(A\right),A+x=\left\{ a+x:a\in A\right\} \in\mathcal{A}.$
\end{defn}

Let us now mention some useful theorem from \cite{key-5} for proof
of our main theorem.
\begin{thm}
Let $\left(S,+\right)$ be a commutative adequate partial semigroup
and let $k\in\mathbb{N}$. If $\left\{ \left(v_{\alpha}^{\left(i\right)}\right)_{\alpha\in\mathcal{P}_{f}\left(\mathbb{N}\right)}:1\leq i\leq k\right\} $
is an adequate set of VIP systems in $S$ , and $\beta\in\mathcal{P}_{f}\left(\mathbb{N}\right)$,
then the family 
\[
\mathcal{A}=\left\{ \begin{array}{c}
\left\{ a,a+v_{\alpha}^{\left(1\right)},a+v_{\alpha}^{\left(2\right)},...,a+v_{\alpha}^{\left(k\right)}\right\} :\\
\alpha\in\mathcal{P}_{f}\left(\mathbb{N}\right),a\in\sigma\left(\left\{ v_{\alpha}^{\left(1\right)},v_{\alpha}^{\left(2\right)},...,v_{\alpha}^{\left(k\right)}\right\} \right)and\text{ }\alpha>\beta
\end{array}\right\} 
\]
 is adequately partition regular.
\end{thm}

\begin{proof}
\cite{key-5}, Theorem $3.7$
\end{proof}
\begin{thm}
Let $\left(S,+\right)$ be a commutative adequate partial semigroup
and let $\mathcal{A}$ be a shift invariant ,adequately partition
regular family of finite subsets of S. Let $E\subseteq S$ be piecewise
syndetic. Then $E$ contains a member of $\mathcal{A}$.
\end{thm}

\begin{proof}
\cite{key-5},Theorem $3.8$
\end{proof}
\begin{thm}
Let $\left\{ \left(v_{\alpha}^{\left(i\right)}\right)_{\alpha\in\mathcal{P}_{f}\left(\mathbb{N}\right)}:1\leq i\leq k\right\} $
be an adequate set of VIP systems and pick $d,t\in\mathbb{N}$, a
set $\left\{ \left(m_{\gamma}^{\left(i\right)}\right)_{\gamma\in\mathcal{F}_{d}}:1\leq i\leq k\right\} $,
a set of VIP systems
\end{thm}

$\left\{ \left(u_{\alpha}^{\left(i\right)}=\sum_{\gamma\subseteq\alpha,\gamma\in\mathcal{F}_{d}}n_{\gamma}^{\left(i\right)}\right)_{\alpha\in\mathcal{P}_{f}\left(\mathbb{N}\right)}:1\leq i\leq t\right\} $,

and sets $E_{1},E_{2},...,E_{k}\subseteq\left\{ 1,2,...,t\right\} $
satisfying conditions $\left(1\right),\left(2\right),$and $\left(3\right)$
of Definition 3.5. Let $\alpha_{1},\alpha_{2},...,\alpha_{s}\in\mathcal{P}_{f}\left(\mathbb{N}\right)$
with $\alpha_{1}<\alpha_{2}<...<\alpha_{s}$. For $F\subseteq\left\{ 1,2,...,s\right\} $,
$i\in\left\{ 1,2,...,k\right\} $ and $\varphi\in\mathcal{F}_{d}$
with $\varphi>\alpha_{s}$, and $1\leq i\leq k$, let

$b_{\varphi}^{\left(i,F\right)}=\sum_{\psi\subseteq\cup_{j\in F}\alpha_{j},\mid\psi\mid\leq d-\mid\varphi\mid}m_{\varphi\cup\psi}^{\left(i\right)}$.

For $F\subseteq\left\{ 1,2,...,s\right\} $, $i\in\left\{ 1,2,...,k\right\} $,
and $\beta\in\mathcal{F}_{d}$ with $\beta>\alpha_{s}$, let

$q_{\beta}^{\left(i,F\right)}=\sum_{\varphi\subseteq\beta,\varphi\in\mathcal{F}_{d}}b_{\varphi}^{\left(i,F\right)}$.

Then $\left\{ \left(q_{\beta}^{\left(i,F\right)}\right)_{\beta\in\mathcal{P}_{f}\left(\mathbb{N}\right),\beta>\alpha_{s}}:i\in\left[k\right],F\subseteq\left\{ 1,2,...,s\right\} \right\} $
is an adequate set of VIP systems.
\begin{proof}
\cite{key-5}, Theorem 3.10.
\end{proof}
Here is our main theorem of this section. 
\begin{thm}
\label{main}Let $\left(S,+\right)$ be commutative adequate partial
semigroup and $p$ be an idempotent in $K\left(\delta S\right)$and
let $\left(C_{n}\right)_{n=1}^{\infty}$ be a sequence of members
of $p$ and 
\[
\left\{ \left(v_{\alpha}^{\left(i\right)}\right)_{\alpha\in\mathcal{P}_{f}\left(\mathbb{N}\right)}:1\leq i\leq k\right\} 
\]
 be $k$-many adequate set of VIP system. Then there exists sequences
$\left(a_{n}\right)_{n=1}^{\infty}$in $S$ and $\left(\alpha_{n}\right)_{n=1}^{\infty}$in
$\mathcal{P}_{f}\left(\mathbb{N}\right)$ such that $\alpha_{n}<\alpha_{n+1}$
for every $n$ and for every $F\in\mathcal{P}_{f}\left(\mathbb{N}\right),$$\gamma=\cup_{t\in F}\alpha_{t}$
such that for $m=minF$
\end{thm}

$\left\{ \sum_{t\in F}a_{t}\right\} \cup\left\{ \sum_{t\in F}a_{t}+v_{\gamma}^{\left(i\right)}:1\leq i\leq k\right\} \subseteq C_{m}.$
\begin{proof}
Let $C_{n}\in p\in K\left(\delta S\right).$We assume $C_{n+1}\subseteq C_{n}$
for each $n\in\mathbb{N}$ ( If not, consider $B_{n}=\cap_{i=1}^{n}C_{i}$
, so $B_{n+1}\subseteq B_{n}$ ). Let for each $n\in\mathbb{N}$,
let 
\[
C_{n}^{*}=\left\{ x\in C_{n}:-x+C_{n}\in p\right\} .
\]
Then for each $x\in C_{n}^{*}$ , $-x+C_{n}^{*}\in p$ by lemma 2.12.
Let 
\[
\mathcal{A}=\left\{ \left\{ a,a+v_{\alpha}^{\left(1\right)},a+v_{\alpha}^{\left(2\right)},...,a+v_{\alpha}^{\left(k\right)}\right\} :\alpha\in\mathcal{P}_{f}\left(\mathbb{N}\right),a\in\sigma\left(\left\{ v_{\alpha}^{\left(1\right)},v_{\alpha}^{\left(2\right)},...,v_{\alpha}^{\left(k\right)}\right\} \right)\right\} 
\]
Then by theorem 4.3 $\mathcal{A}$ is adequately partition regular
and $\mathcal{A}$ is trivially shift invariant. Since for each $n\in\mathbb{N},C_{n}^{*}\in p$
and $p\in K\left(\delta S\right)$ ,$C_{n}^{*}$ is piecewise syndetic.
So by theorem 4.4 , for some $a_{1}\in S$ and $\alpha_{1}\in\mathcal{P}_{f}\left(\mathbb{N}\right)$
such that 
\[
\left\{ a_{1},a_{1}+v_{\alpha_{1}}^{\left(1\right)},a_{1}+v_{\alpha_{1}}^{\left(2\right)},...,a_{1}+v_{\alpha_{1}}^{\left(k\right)}\right\} \subseteq C_{n}^{*}
\]
 for every $n\in\mathbb{N}.$ Now the proof is by induction, let $n\in\mathbb{N}$
and assume that we have chosen $\left(a_{t}\right)_{t=1}^{n}$in $S$
and $\left(\alpha_{t}\right)_{t=1}^{n}$ in $\mathcal{P}_{f}\left(\mathbb{N}\right)$
such that 

$\left(1\right)$ for $t\in\left[n-1\right]$, if any, $\alpha_{t}<\alpha_{t+1}$,
and

$\left(2\right)$ for $\phi\neq F\subseteq\left[n\right]$ $minF=m$
, if $\gamma=\cup_{t\in F}\alpha_{t}$ , then $\sum_{t\in F}a_{t}\in C_{m}^{*}$
and for each $i\in\left[k\right]$ , $\sum_{t\in F}a_{t}+v_{\gamma}^{\left(i\right)}\in C_{m}^{*}$. 

For each $\gamma\in FU\left(\left(\alpha_{t}\right)_{t=1}^{n}\right)$
and each $i\in\left[k\right]$, let 
\[
\left(q_{\beta}^{\left(i,\gamma\right)}\right)_{\beta\in\mathcal{P}_{f}\left(\mathbb{N}\right)}=\left(v_{\gamma\cup\beta}^{\left(i\right)}-v_{\gamma}^{\left(i\right)}\right)_{\beta\in\mathcal{P}_{f}\left(\mathbb{N}\right),\beta>\alpha_{n}}
\]
 By theorem 4.5 , the family,

$\left\{ \left(q_{\beta}^{\left(i,\gamma\right)}\right)_{\beta\in\mathcal{P}_{f}\left(\mathbb{N}\right),\beta>\alpha_{n}}:i\in\left[k\right],\gamma\in FU\left(\left(\alpha_{t}\right)_{t=1}^{n}\right)\right\} \cup\left\{ \left(v_{\beta}^{\left(i\right)}\right)_{\beta\in\mathcal{P}_{f}\left(\mathbb{N}\right)}:i\in\left[k\right]\right\} $
is an adequate set of VIP systems. Let 
\[
\mathcal{B}=\left\{ \begin{array}{c}
\left\{ a\right\} \cup\left\{ a+v_{\alpha}^{\left(i\right)}:i\in\left[k\right]\right\} \cup\cup_{\gamma\in FU\left(\left(\alpha_{t}\right)_{t=1}^{n}\right)}\left\{ a+q_{\alpha}^{\left(i,\gamma\right)}:i\in\left[k\right]\right\} \\
\begin{array}{c}
:\alpha\in\mathcal{P}_{f}\left(\mathbb{N}\right),\alpha<\alpha_{n}\text{\,}and\\
\text{ }a\in\sigma\left(\left\{ v_{\alpha}^{\left(i\right)}:i\in\left[k\right]\right\} \cup\left\{ q_{\alpha}^{\left(i,\gamma\right)}:i\in\left[k\right],\gamma\in FU\left(\left(\alpha_{t}\right)_{t=1}^{n}\right)\right\} \right)
\end{array}
\end{array}\right\} 
\]
Then by theorem $4.3$, $\mathcal{B}$ is adequately partition regular.
Let 
\[
D=C_{n+1}^{*}\cap\bigcap_{m=1}^{n}\left[\begin{array}{c}
\cap\left\{ -\sum_{t\in H}a_{t}+C_{m}^{*}:m=minH,\phi\neq H\subseteq\left[n\right]\right\} \bigcap\\
\cap\left\{ \begin{array}{c}
-\left(\sum_{t\in H}a_{t}+v_{\gamma}^{\left(i\right)}\right)+C_{m}^{*}:\\
m=minH,\phi\neq H\subseteq\left[n\right],and\gamma=\sum_{t\in H}\alpha_{t}
\end{array}\right\} 
\end{array}\right]
\]
Then $D\in p$ and $D$ is piecewise syndetic. So by theorem 3.8,
for some $\alpha_{n+1}\in\mathcal{P}_{f}\left(\mathbb{N}\right)$
such that $\alpha_{n+1}>\alpha_{n}$ and some 
\[
a_{n+1}\in\sigma\left(\left\{ v_{\alpha_{n+1}}^{\left(i\right)}:i\in\left[k\right]\right\} \cup\left\{ q_{\alpha_{n+1}}^{\left(i,\gamma\right)}:i\in\left[k\right]\text{ and }\gamma\in FU\left(\left(\alpha_{t}\right)_{t=1}^{n}\right)\right\} \right)
\]
 such that 
\[
\left\{ a_{n+1}\right\} \cup\left\{ a_{n+1}+v_{\alpha_{n+1}}^{\left(i\right)}:i\in\left[k\right]\right\} \cup\cup_{\gamma\in FU\left(\left(\alpha_{t}\right)_{t=1}^{n}\right)}\left\{ a_{n+1}+q_{\alpha_{n+1}}^{\left(i,\gamma\right)}:i\in\left[k\right]\right\} \subseteq D.
\]
 By induction hypothesis $\left(1\right)$ trivially holds. To verify
$\left(2\right)$, let $\phi\neq F\subseteq\left[n+1\right]$ and
let $\gamma=\cup_{t\in F}\alpha_{t}$ . If $n+1\notin F,$ the condition
holds by assumption. If $F=\left\{ n+1\right\} $, then we have 
\[
\left\{ a_{n+1}\right\} \cup\left\{ a_{n+1}+v_{\alpha_{n+1}}^{\left(i\right)}:i\in\left[k\right]\right\} \subseteq D\subseteq C_{n+1}^{*}.
\]
So, let assume $\left\{ n+1\right\} \subsetneq F$, let $H=F$ \textbackslash$\left\{ n+1\right\} $,
and let $\mu=\cup_{t\in F}\alpha_{t}$ . Then $a_{n+1}\in D\subseteq-\sum_{t\in H}a_{t}+C_{m}^{*}$,
where $m=minH$.

Let $\gamma=\sum_{t\in H}\alpha_{t}$ and let $i\in\left[k\right]$
. Then $a_{n+1}+q_{\alpha_{n+1}}^{\left(i,\gamma\right)}\in D\subseteq-\left(\sum_{t\in H}a_{t}+v_{\gamma}^{\left(i\right)}\right)+C_{m}^{*}$
, $m=minH$ 

and so $\left(\sum_{t\in H}a_{t}+v_{\gamma}^{\left(i\right)}\right)+\left(a_{n+1}+q_{\alpha_{n+1}}^{\left(i,\gamma\right)}\right)\in C_{m}^{*}$
. That is 
\[
\sum_{t\in F}a_{t}+v_{\mu}^{\left(i\right)}=\left(\sum_{t\in H}a_{t}+a_{n+1}\right)+\left(v_{\gamma}^{\left(i\right)}+q_{\alpha_{n+1}}^{\left(i,\gamma\right)}\right)\in C_{m}^{*}\subseteq C_{m}.
\]
\end{proof}
\begin{thm}
Let $\left(S,+\right)$ be a commutative adequate partial semigroup
and let $\left(C_{n}\right)_{n=1}^{\infty}$be sequence of central
sets where $C_{n}\subseteq S$. Suppose that 
\[
\left\{ \left(v_{\alpha}^{\left(i\right)}\right)_{\alpha\in\mathcal{P}_{f}\left(\mathbb{N}\right)}:i\in\left[k\right]\right\} 
\]
is an adequate set of VIP systems. Then there exists an IP ring $\mathcal{F}^{\left(1\right)}$
and an IP system $\left(b_{\alpha}\right)_{\alpha\in\mathcal{F}^{\left(1\right)}}$
in $S$ such that $\mathcal{F}^{\left(1\right)}=FU\left(\left(\alpha_{n}\right)_{n=1}^{\infty}\right)$
where $\left(\alpha_{n}\right)_{n=1}^{\infty}$ is a sequence of members
of $\mathcal{P}_{f}\left(\mathbb{N}\right)$ such that $max\alpha_{n}<min\alpha_{n+1}$
for all $n\in\mathbb{N}$ and for all $\alpha\in\mathcal{F}^{\left(1\right)}$,
where $\alpha=\cup_{t\in F}\alpha_{t}$, $F\in\mathcal{P}_{f}\left(\mathbb{N}\right)$
and $minF=m$, $\left\{ b_{\alpha},b_{\alpha}+v_{\alpha}^{\left(i\right)},...,b_{\alpha}+v_{\alpha}^{\left(k\right)}\right\} \subseteq C_{m}$.
\end{thm}

\begin{proof}
Choose $\left(a_{n}\right)_{n=1}^{\infty}$ and $\left(\alpha_{n}\right)_{n=1}^{\infty}$as
in theorem 4.6. Put $b_{\alpha}=\sum_{t\in F}a_{t}$. 
\end{proof}
\begin{thm}
Let $\left(S,+\right)$ be a commutative adequate partial semigroup
and let $C_{n}\subseteq S$, $n\in\mathbb{N}$ be central sets. Suppose
that $\left\{ \left(v_{\alpha}^{\left(i\right)}\right)_{\alpha\in\mathcal{P}_{f}\left(\mathbb{N}\right)}:i\in\left[k\right]\right\} $
is an adequate set of VIP systems. Then there exists sequences $\left(a_{n}\right)_{n=1}^{\infty}$in
$S$ and $\left(\alpha_{n}\right)_{n=1}^{\infty}$ in $\mathcal{P}_{f}\left(\mathbb{N}\right)$
such that $\alpha_{n}<\alpha_{n+1}$ for each $n$ and such that for
every $F\in\mathcal{P}_{f}\left(\mathbb{N}\right)$, $\sum_{t\in F}a_{t}\in C_{m}$
where $m=min$F and if $\beta_{1}<\beta_{2}<...<\beta_{s}$, where
each $\beta_{j}\subseteq F$ and $i_{1},i_{2},...,i_{s}\in\left\{ 1,2,...,k\right\} $
then writing $\gamma_{j}=\cup_{t\in\beta_{j}}\alpha_{t}$ for $j\in\left\{ 1,2,...,s\right\} $we
have $\sum_{t\in F}a_{t}+\sum_{j=1}^{s}v_{\gamma_{j}}^{\left(i_{j}\right)}\in C_{m}$.
\end{thm}

\begin{proof}
To prove we will modify the induction hypothesis $\left(2\right)$
of the proof of theorem 3.11 by $\left(2\right)$ for $\phi\neq F\subseteq\left[n\right]$,
$minF=m$ ,$\sum_{t\in F}a_{t}\in C_{m}^{*}$ and if $\beta_{1}<\beta_{2}<...<\beta_{s}$,
where each $\beta_{j}\subseteq F$ and $i_{1},i_{2},...,i_{s}\in\left[k\right]$
and for $j\in\left\{ 1,2,...,s\right\} $$\gamma_{j}=\cup_{t\in\beta_{j}}\alpha_{t}$
then $\sum_{t\in F}a_{t}+\sum_{j=1}^{s}v_{\gamma_{j}}^{\left(i_{j}\right)}\in C_{m}^{*}$.
We have to change the set $D$ in the proof of theorem 3.11 by
\[
D=C_{n+1}^{*}\bigcap\cap_{m=1}^{n}\left[\begin{array}{c}
\cap\left\{ -\sum_{t\in H}a_{t}+C_{m}^{*}:\phi\neq H\subseteq\left[n\right],minH=m\right\} \bigcap\\
\left\{ \begin{array}{c}
-\left(\sum_{t\in H}a_{t}+\sum_{j=1}^{s}v_{\gamma_{j}}^{\left(i_{j}\right)}\right)+C_{m}^{*}:\\
\begin{array}{c}
\phi\neq H\subseteq\left[n\right],minH=m,s\in\mathbb{N},\beta_{1}<\beta_{2}<...<\beta_{s},\\
\cup_{j=1}^{s}\beta_{j}\subseteq H\text{ }and\text{ }forj\in\left[s\right],\gamma_{j}=\cup_{t\in\beta_{j}}\alpha_{t}
\end{array}
\end{array}\right\} 
\end{array}\right]
\]
rest of the proof is quite similar to the proof of theorem 3.11 so
we skip that part. Here we speak few words about weak VIP systems.
If $S$ be a commutative and cancellative semigroup then $S$ can
be embedded in a group this group is called group of quotients. 
\end{proof}
\begin{defn}
Let $\left(S,+\right)$ be a commutative cancellative semigroup and
let $G$ be the group of quotients of $S$. A sequence $\left(v_{\alpha}\right)_{\alpha\in\mathcal{P}_{f}\left(\mathbb{N}\right)}$
in $S$ is called a weak VIP systems if it is a VIP system in $G.$
\end{defn}

\begin{cor}
Let $\left(S,+\right)$ be a commutative cancellative semigroup and
let $C_{n}\subseteq S$, $n\in\mathbb{N}$ be central sets, and let
$\left\{ \left(v_{\alpha}^{\left(i\right)}\right)_{\alpha\in\mathcal{P}_{f}\left(\mathbb{N}\right)}:i\in\left[k\right]\right\} $
be a set of weak VIP systems in S. Then there exists sequences $\left(a_{n}\right)_{n=1}^{\infty}$in
$S$ and $\left(\alpha_{n}\right)_{n=1}^{\infty}$ in $\mathcal{P}_{f}\left(\mathbb{N}\right)$
such that $\alpha_{n}<\alpha_{n+1}$ for each $n$ and such that for
every $F\in\mathcal{P}_{f}\left(\mathbb{N}\right)$ and every $i\in\left[k\right]$,
if $\gamma=\cup_{t\in F}\alpha_{t}$, then $\sum_{t\in F}a_{t}+v_{\gamma}^{\left(i\right)}\in C_{m}$
,where $m=\min F$ .
\end{cor}

\begin{proof}
Let $G$ be the group of quotients of $S.$Then, with substraction
in $G$, we have $G=\left\{ a-b:a,b\in S\right\} $. We claim that
$S$ is piecewise syndetic in $G.$ That is there exists $H\in\mathcal{P}_{f}\left(G\right)$
such that for each $F\in\mathcal{P}_{f}\left(G\right)$, there exists
$x\in G$ such that $F+x\subseteq\cup_{t\in H}\left(-t+S\right)$.
Indeed , let $H=\left\{ 0\right\} $ and let $F\in\mathcal{P}_{f}\left(G\right)$
be given. Pick $l\in\mathbb{N}$ and 
\[
a_{1},a_{2},...,a_{l},b_{1},b_{2},...,b_{l}
\]
 in $S$ such that $F=\left\{ a_{i}-b_{i}:i\in\left[l\right]\right\} $.
Let $x=\sum_{i=1}^{l}b_{i}$. Then $F+x\subseteq S=-0+S.$ Since $S$
is piecewise syndetic, $\overline{S}\cap K\left(\beta G\right)\neq\phi$
by $\left[\left[5\right],Theo\text{rem }4.40\right]$ and consequently
$K\left(\beta S\right)=\overline{S}\cap K\left(\beta G\right)$ by
$\text{\ensuremath{\left[\text{\ensuremath{\left[5\right]}},Theorem\text{ }1.65\right]}. }$Since
$C_{n}$ are central in S, by definition there is some idempotent
$p\in K\left(\beta S\right)$ such that $C_{n}\in P.$ But then $p\in K\left(\beta G\right)$
and thus $C_{n}$ are central in G. Also, for each $i\in\left[k\right]$,
$\left(v_{\alpha}^{\left(i\right)}\right)_{\alpha\in\mathcal{P}_{f}\left(\mathbb{N}\right)}$
is a weak VIP system in $S$ and is therefore a VIP system in $G$.
Thus, $\left\{ \left(v_{\alpha}^{\left(i\right)}\right)_{\alpha\in\mathcal{P}_{f}\left(\mathbb{N}\right)}:i\in\left[k\right]\right\} $
is an adequate set of VIP systems in $G$ so by theorem3.11, there
exists sequences $\left(a_{n}\right)_{n=1}^{\infty}$in G and $\left(\alpha_{n}\right)_{n=1}^{\infty}$
in $\mathcal{P}_{f}\left(\mathbb{N}\right)$ such that $\alpha_{n}<\alpha_{n+1}$
for each $n$ and such that for every $F\in\mathcal{P}_{f}\left(\mathbb{N}\right)$,
if $\gamma=\cup_{t\in F}\alpha_{t}$, then $\left\{ \sum_{t\in F}a_{t}\right\} \cup\left\{ \sum_{t\in F}a_{t}+v_{\gamma}^{\left(i\right)}:i\in\left[k\right]\right\} \subseteq C_{m}.$
In particular, each $a_{t}$ is in $C_{m}\subseteq S$ so $\left(a_{n}\right)_{n=1}^{\infty}$
is a sequence in S as required.
\end{proof}
Here we present the ``VIP-Free'' version of Theorem \ref{main}
and a similar proof. 
\begin{thm}
Let $\left(S,+\right)$ be a commutative adequate partial semigroup
and let $U$ be a set, and for each $s\in U$, let $T_{s}$ be a set.
For each $s\in U$ and each $t\in T_{s}$, let $A_{s,t}\in\mathcal{P}_{f}\left(S\right),$such
that the family $\mathcal{A}_{s}=\left\{ A_{s,t}:t\in T_{s}\right\} $
is shift invariant and adequately partition regular. Let $s_{1}\in U$
and suppose $\phi:\cup_{s\in U}\left(\left\{ s\right\} \times T_{s}\right)\to U$
is a function. If $C_{n}\in S,n\in\mathbb{N}$ is a sequence of central
set then there exists sequences $\left(s_{n}\right)_{n=2}^{\infty}$
in $U$ and $\left(t_{n}\right)_{n=1}^{\infty}$ with each $t_{n}\in T_{s_{n}}$
such that 

$\phi\left(s_{n-1},t_{n-1}\right)=s_{n}$ for $n\geq2$ and such that
if $n_{1}<n_{2}<...<n_{m}$ and for each $i\in\left\{ 1,2,...,m\right\} ,x_{n_{i}}\in A_{s_{n_{i}},t_{n_{i}}}$,
then 
\[
\left(x_{n_{1}}+x_{n_{2}}+...+x_{n_{m}}\right)\in C_{n_{1}}\left(\text{the sum is defined}\right).
\]
\end{thm}

\begin{proof}
The proof of Theorem \ref{main} is modified. Having choosen $\left(s_{i}\right)_{i=1}^{n}$
and $\left(t_{i}\right)_{i=1}^{n-1}$, replace the adequately partition
regular family $\mathcal{B}$ constructed in the proof of Theorem
\ref{main} by $\mathcal{A}_{s_{n}}$ and replace the piecewise syndetic
set $D$ by

$\hat{D}=C_{n+1}^{*}\cap\cap\left[\left\{ \begin{array}{c}
-\left(x_{n_{1}}+x_{n_{2}}+...+x_{n_{m}}\right)+C_{n_{1}}^{*}:\\
n_{1}<n_{2}<...<n_{m}<n\text{ and each \ensuremath{x_{n_{i}}\in A_{s_{n_{i}},t_{n_{i}}}}}
\end{array}\right\} \right]$

Then one chooses $t_{n}$ so that $A_{s_{n},t_{n}}\subseteq\hat{D}$
and let $s_{n+1}=\phi\left(s_{n},t_{n}\right).$
\end{proof}

\section{APPLICATION}

Here we briefly discuss about the application of Theorem \ref{main}$.$ 
\begin{thm}
Let $\left(S,+\right)$ be a commutative semigroup and let $C_{n}\subseteq S$,
$n\in\mathbb{N}$ be central sets. Suppose that $\left\{ \left(v_{\alpha}^{\left(i\right)}\right)_{\alpha\in\mathcal{P}_{f}\left(\mathbb{N}\right)}:i\in\left[k\right]\right\} $
is a set of IP systems. Then there exists sequences $\left(a_{n}\right)_{n=1}^{\infty}$in
$S$ and $\left(\alpha_{n}\right)_{n=1}^{\infty}$ in $\mathcal{P}_{f}\left(\mathbb{N}\right)$
such that $\alpha_{n}<\alpha_{n+1}$ for each $n$ and such that for
every $F\in\mathcal{P}_{f}\left(\mathbb{N}\right)$, $\sum_{t\in F}a_{t}\in C_{m}$
where $m=min$F and if $\beta_{1}<\beta_{2}<...<\beta_{s}$, where
each $\beta_{j}\subseteq F$ and $i_{1},i_{2},...,i_{s}\in\left\{ 1,2,...,k\right\} $
then writing $\gamma_{j}=\cup_{t\in\beta_{j}}\alpha_{t}$ for $j\in\left\{ 1,2,...,s\right\} $
we have $\sum_{t\in F}a_{t}+\sum_{j=1}^{s}v_{\gamma_{j}}^{\left(i_{j}\right)}\in C_{m}$.
\end{thm}

\begin{proof}
In a commutative semigroup any set of IP system is an adequate set
of VIP system, so theorem $4.8$ applies.
\end{proof}
\begin{defn}
Let $l\in\mathbb{N}$ , a set-monomial $\left(\text{over }\mathbb{N}^{l}\right)$
in the variable $X$ is an expression $m\left(X\right)=S_{1}\times S_{2}\times\ldots\times S_{l}$
, where for each $i\in\left\{ 1,2,\ldots,l\right\} $, $S_{i}$ is
either the symbol $X$ or a nonempty singleton subset of $\mathbb{N}$
(these are called coordinate coefficients). The degree of the monomial
is the number of times the symbol $X$ appears in the list $S_{1},\ldots,S_{l}$.
For example, taking $l=3$, $m\left(X\right)=\left\{ 5\right\} \times X\times X$
is a set-monomial of degree 2, while $m\left(X\right)=X\times\left\{ 17\right\} \times\left\{ 2\right\} $
is a set-monomial of degree 1. A\emph{ set-polynomial} is an expression
of the form $p\left(X\right)=m_{1}\left(X\right)\cup m_{2}\left(X\right)\cup\ldots\cup m_{k}\left(X\right)$,
where $k\in\mathbb{N}$ and $m_{1}\left(X\right)\cup m_{2}\left(X\right)\cup\ldots\cup m_{k}\left(X\right)$
are set-monomials. The degree of a set-polynomial is the largest degree
of its set-monomial ``summands'' , and its constant term consists
of the ``sum'' of those $m_{i}$ that are constant, i.e., of degree
zero. 
\end{defn}

\begin{lem}
Let $l\in\mathbb{N}$ and let $\mathcal{P}$ be a finite family of
set polynomial over 
\[
\left(\mathcal{P}_{f}\left(\mathbb{N}^{l}\right),+\right)
\]
 whose constant terms are empty. Then there exists $q\in\mathbb{N}$
and an IP ring $\mathcal{F}^{\left(1\right)}=\left\{ \alpha\in\mathcal{P}_{f}\left(\mathbb{N}\right):min\alpha>q\right\} $
such that $\left\{ \left(P\left(\alpha\right)\right)_{\alpha\in\mathcal{F}^{\left(1\right)}}:P\left(X\right)\in\mathcal{P}\right\} $
is an adequate set of VIP systems.
\end{lem}

\begin{proof}
\cite{key-5}, Lemma 4.3
\end{proof}
\begin{thm}
Let $l\in\mathbb{N}$ and let $\mathcal{P}$ be a finite family of
set polynomial over $\left(\mathcal{P}_{f}\left(\mathbb{N}^{l}\right),+\right)$
whose constant terms are empty. If $C_{n}\subseteq\mathcal{P}_{f}\left(\mathbb{N}^{l}\right),n\in\mathbb{N}$
are central sets then there exists sequences $\left(A_{n}\right)_{n=1}^{\infty}$
in $\mathcal{P}_{f}\left(\mathbb{N}^{l}\right)$ and $\left(\alpha_{n}\right)_{n=1}^{\infty}$such
that $\alpha_{n}<\alpha_{n+1}$ for each $n$ and such that for every
$F\in\mathcal{P}_{f}\left(\mathbb{N}\right)$. We have $\left\{ A_{\gamma}\right\} \cup\left\{ A_{\gamma}+P\left(\gamma\right):P\in\mathcal{P}\right\} \subseteq C_{m}$,
where $m=minF$,$\gamma=\cup_{t\in F}\alpha_{t}$ and $A_{\gamma}=\sum_{t\in F}A_{t}$.
\end{thm}

\begin{proof}
By lemma $5.3$ there is an IP ring $\mathcal{F}^{\left(1\right)}$
such that $\left\{ \left(P\left(\alpha\right)\right)_{\alpha\in\mathcal{F}^{\left(1\right)}}:P\left(X\right)\in\mathcal{P}\right\} $
is an adequate set of VIP systems. Thus Theorem \ref{main} applies.
\end{proof}
$\vspace{0.3in}$

\textbf{Acknowledgment:} The first author acknowledge the Grant CSIR-UGC
NET fellowship with file No. 09/106(0202)/2020-EMR-I. The second author
acknowledge the support from University Research Scholarship of University
of Kalyani with id-1F-7/URS/Mathematics/2023/S-502.

$\vspace{0.3in}$

\end{document}